\documentclass[final]{elsarticle}
\usepackage{epsfig}
\usepackage[colorlinks=true]{hyperref}
\usepackage{amsmath,amssymb,amsthm,mathrsfs}

\newtheorem{Lemma}{Lemma}
\newtheorem{Theorem}[Lemma]{Theorem}
\newtheorem{Corollary}{Corollary}

\newtheorem{example}{Example}

\newcommand{\set}[1]{\left\{#1\right\}}

\newcommand{\bmt}{\left[ \begin{array}{ccccccccc}}
\newcommand{\emt}{\end{array}\right]}
\newcommand{\bea}{\begin{eqnarray}}
\newcommand{\eea}{\end{eqnarray}}
\newcommand{\bean}{\begin{eqnarray*}}
\newcommand{\eean}{\end{eqnarray*}}

\newcommand{\IC}{\mathbb C}
\newcommand{\IR}{\mathbb R}

\newcommand{\Cnn}{\IC^{n\times n}}
\newcommand{\Cn}{\IC^n}
\newcommand{\Rnn}{\IR^{n\times n}}

\newcommand{\tr}{\textnormal{tr}}
\newcommand{\rank}{\textnormal{rank}}
\newcommand{\tn}[1]{\quad\textnormal{#1}\quad}
\newcommand{\bm}{\boldsymbol}
\newcommand{\F}{\mathcal{F}}

\newcommand{\Rk}{\mathsf{P}}

\newcommand{\lam}[1]{\lambda_}

\newcommand{\bmat}[1]{\begin{bmatrix}#1\end{bmatrix}}

\journal{Linear Algebra and its Applications}
\begin{document}

\begin{frontmatter}
\title{Rank-one Characterization of Joint Spectral Radius of Finite Matrix Family\tnoteref{t1}}
\tnotetext[t1]{This work was supported by NSF 1021203 of the United States.}

\author[LX]{Jun Liu}
\ead{jliu@math.siu.edu}
\author[LX]{Mingqing Xiao}
\ead{mxiao@math.siu.edu}


\address[LX]{Department of Mathematics, Southern Illinois University, Carbondale, IL 62901-4408, USA}

\begin{abstract}
In this paper we study the joint/generalized spectral radius of a finite set of matrices in terms of its rank-one approximation by singular value decomposition. In the first part of the paper, we show that any finite set of matrices with at most one element's rank being  greater than one satisfies the finiteness property under 
the framework of (invariant) extremal norm. 
Formula for the computation of joint/generalized spectral radius for this class of matrix family 
is derived. Based on that, in the second part,  we further study the joint/generalized spectral radius of finite sets of 
general matrices through constructing rank-one approximations in terms of 
singular value decomposition, and some new characterizations of joint/generalized spectral radius 
are obtained. Several benchmark examples from applications as well as corresponding numerical computations
are provided to illustrate the approach.
\end{abstract}


\end{frontmatter}

\pagestyle{myheadings}
\thispagestyle{plain}

\markboth{J. Liu and M. Xiao}{Rank-one Characterization of Joint Spectral Radius}

\section{Introduction}
The joint spectral radius of a finite set of $n\times n$ matrices describes the maximal asymptotic growth rate of products of matrices taken in the set, and it plays a critical role in many applications, for example, in the study of wavelet theory \cite{Broker2000,Daubechies1992LAA,Daubechies1992SIAM,Daubechies2001,Protasov2006},
stability of switched and hybrid systems \cite{Dai2008,Dai2010, Gurvits1995,Shorten2007}, 
subdivision algorithms for generating curves \cite{Derfel1995,Guglielmi2010},
overlap-free words \cite{Jungers2009},
asymptotic behavior of partition functions \cite{Protasov2000}, 
and their references therein. 
Therefore, efficient algorithms with desirable accuracy are necessary for the computation of joint spectral radius in order to meet the demands from applications.

Let a finite set  $\F=\set{A_1,A_2,\cdots,A_m}\subset \IC^{n\times n}$ of complex $n\times n$ matrices be given and we denote by  $\F_k$ the set of all possible products of length $k\ge 1$ with elements from $\F$, i.e., 
$$\F_k=\set{A_{i_1}A_{i_2}\cdots A_{i_k}:\; A_{i_j}\in \F; 1\le i_j\le m,\ j=1,\ldots,k}.$$
Let $\|\cdot\|$ be any sub-multiplicative matrix norm and $\rho(A)$ be the spectral radius of a matrix $A$.
The joint spectral radius (JSR) of $\F$, introduced by Rota and Strang \cite{Rota1960}, is defined as
\[
\hat{\rho}(\F)=\lim_{k\to\infty}\max_{A\in\F_k}\|A\|^{1/k},
\]
and the generalized spectral radius of $\F$, initiated by Daubechies and Lagarias \cite{Daubechies1992LAA}, is given by
\[
\bar{\rho}(\F)=\limsup_{k\to\infty}\max_{A\in\F_k}\rho(A)^{1/k}.
\]
Since the equality $\hat{\rho}(\F)=\bar{\rho}(\F)$ has been established 
for any finite set of matrices \cite{Berger1992,Elsner1995,Shih1997}, unless it is necessary, 
we shall not distinguish between them and 
designate an unified notation $\rho(\F)(=\hat{\rho}(\F)=\bar{\rho}(\F))$ throughout the paper.
Another equivalent variational way of characterizing JSR is \cite{Rota1960}
\bea
\label{infnorm}
\rho(\F)=\inf_{\|\cdot\|}\max_{A\in\F}\|A\|,
\eea
where the infimum is taken over the set of all sub-multiplicative matrix norms.
Whenever the infimum in (\ref{infnorm}) is attained (thus it is a minimum),
the corresponding norm $\|\cdot\|_*$ will be called an extremal norm of $\F$ \cite{Wirth2002}.
The characteristic (\ref{infnorm}) is important and useful if $\|\cdot\|_*$ is available and efficiently computable
for a given $\F$.

Earlier algorithms \cite{Gripenberg1996,Maesumi1996} for computing or approximating the joint/generalized spectral radius mostly make use of
the following inequalities
\begin{equation}
\label{3ineq}
 \max_{A\in\F_k}\rho(A)^{1/k}\le \rho(\F) \le \max_{A\in\F_k}\|A\|^{1/k}
\end{equation}
for any $k\ge 1$. 
In general, however,  such a brute-force approach is impractical to solve the problem
since the computational cost will easily exceed the limit of toady's computer capacity even for small $k$, 
in particular, for large-scale matrices.
In order to obtain better approximations within current computational capacity, 
many numerical methods were proposed during last decade.
We categorize them into two main approaches.

The first approach is to try to construct the extremal norm $\|\cdot\|_*$ or to approximate it when it exists.
One necessary and sufficient condition for the existence of an extremal norm is
the non-defectiveness of the corresponding normalized matrix family \cite{Guglielmi2005}, 
which is not algorithmically decidable \cite{Blondel2000}.
In \cite{Blondel2005Laa}, the minimization was restricted to the set of ellipsoid norms,
which can be efficiently approximated by current convex optimization algorithms.
This approach provides a theoretical precision estimation of $\rho(\F)$ in limited applicable cases. 
In \cite{Guglielmi2005,Guglielmi2008,Guglielmi2009},
the minimization was confined to the set of complex polytope norms. 
The successful construction of such a polytope norm is not guaranteed in general, 
and it is more suitable to be used to verify  the occurrence of the finiteness property of $\F$  \cite{Lagarias1995},  that is, to validate the case when there is a positive integer $t$ such that 
\[
\rho(\F)=\rho(A_{i_1}A_{i_2}\cdots A_{i_t})^{1/t}
\]
for some finite product $A_{i_1}A_{i_2}\cdots A_{i_t}\in\F_t$, and the corresponding product sequence is called a spectral maximizing sequence.
Within this framework,  other special extremal norms,
such as Barabanov norm \cite{Wirth2002}, Optimal norm \cite{Maesumi2008}, were also considered.
Kozyakin \cite{Kozyakin2010} considered an iterative algorithm which approximates
$\rho(\F)$ through constructing a sequence of approximated Barabanov norms under the assumption of irreducibility.
However, the computational cost is too high since it requires to construct the unit ball with respect to the Barabanov norm, and  the issue of estimating the convergence rate remains unsolved.
The sum of squares method investigated in \cite{Parrilo2008} was intended to approximate
the extremal norm through a multivariate polynomial with norm-like quality under which
the action of matrices becomes contractive. 
However, to obtain an analytic extremal norm expression is usually quite challenging,
and it seems there is no easy solution so far.

The second approach makes use of the cone invariance of a given matrix set $\F$ for computing its JSR when such a property exists  \cite{Protasov1996}.
In \cite{Protasov1996,Protasov2005}, an iterative algorithm which builds an approximated invariant set was developed,
which for a fixed dimension demonstrates polynomial time complexity with respect to $1/\varepsilon$,
where $\varepsilon$ is a given accuracy.
In \cite{Blondel2005}, Blondel and Nesterov introduced a Kronecker lifting based 
approximation to the JSR with an arbitrary accuracy 
under the assumption of the existence of an invariant proper cone, 
which can always be assured via one step of semi-definite lifting with the cost of squaring the matrix dimension.
The exact nature of this cone is irrelevant to the derived accuracy of estimation.
Following this methodology, a new conic programming method was offered in \cite{Protasov2010}, which
gives an improved accuracy estimation by taking the specific nature of the invariant cone into the consideration. In general, the existence of an invariant cone is restrictive and may exclude many interesting cases in real applications. 

In numerical computation of joint/generalized spectral radius, criteria for determining if a given matrix family satisfies the finiteness property
may help us to develop a decidable and efficient algorithm.
The original finiteness conjecture \cite{Lagarias1995} stated that the finiteness property is true for
 any finite $n\times n$ real matrix family, 
which was recently proved to be false \cite{Bousch2002,Blondel2003,Kozyakin2005}.
The existence of such counterexamples
shows that the finiteness property does not hold in general.
At the same time, it has been found since then that many classes of matrices possess this computationally favorable feature.
In \cite{Blondel2005Laa}, Blondel, Nesterov, and Theys proved the finiteness property for
the matrix family with a solvable Lie algebra. 
In \cite{Jungers2009Book}, the normal and commonly triangularizable matrix family were added into the list.
Later on, a restricted version of finite conjecture claimed that the finiteness property is true for
every pair of $n\times n$ sign-matrices \cite{Jungers2008}.
The significance of this conjecture lies in its equivalence to the finiteness property of all sets of rational matrices.
Along this track, the case of $2\times 2$ sign-matrices pair
was proved in \cite{Cicone2010} with the exploration of real extremal polytope norms. However, the similar conclusion for higher dimension remains unknown.
Currently, the list of matrix families satisfying the finiteness property is still very short. 

Rank-one matrices are the simplest class of matrices not only in theoretic analysis 
but also in algorithmic approximations for matrix computation since any matrix can be expressed in terms of the sum of a set of rank-one matrices, for example, in the singular value decomposition (SVD). Gurvits is probably the first one who discusses the  rank-one matrix family in terms of Euclidean norm \cite{Gurvits2005}.   Furthermore, among all those illustrative examples appeared in existing literature related to JSR, 
we have observed that all of the cases with only rank-one matrices satisfy the finiteness property.
This motivates us to consider how to obtain the corresponding JSR and then apply this finding to approximate the JSR of general matrix family.
The main contributions of this paper are: 
(i) by making use of Barabanov norm and rank-one property, 
we show that any finite set of matrices with at most one element's rank being greater than one satisfies the finiteness property and derive the computation formula for its JSR; 
(ii) based on the obtained result in (i), 
we obtain some new characterizations of the JSR of general matrix family in terms of rank-one approximation based on SVD. Numerical computations for some benchmark problems are presented and the results show some favorable estimations over existing algorithms, although we are not able to prove this is always the case.

The paper is organized as follows.
In section 2 we prove that any finite set of matrix family with at most one element's rank being greater than one possesses the finiteness property, and we give some important properties for the computation of its JSR. 
In section 3, we further develop this idea in the study of general matrix family. Connections between the JSR of general finite matrix set and its corresponding JSR of rank-one approximation based on SVD are established.  Further discussions for non-negative matrix family are also presented in this section. Several benchmark examples from real applications as well as their numerical computations are presented in section 4. The paper ends with concluding remarks in section 5. 

\section{Finite rank-one matrix family}

We first give some well-known properties of rank-one matrices,
which will be employed in our subsequent derivations.
Given a matrix $A\in \Cnn$, 
we denote by $\rank(A)$ and $\tr(A)$ the rank and trace of $A$, respectively.
We know from linear algebra that $\rank(A)=1$ if and only if there exist two nonzero vectors $x,y\in \Cn$ such that $A=xy^*$, where $y^*$ denotes the conjugate transpose of $y$.
Obviously, any rank-one matrix $A$ has at most one nonzero eigenvalue, 
denoted by $\lambda(A)=y^*x$.
In particular, the spectral radius of a rank-one matrix $A$ is $\rho(A)=|\lambda(A)|=|\tr(A)|$.
For any two rank-one matrices $A_1=x_1y_1^*\in \Cnn$ and $A_2=x_2y_2^*\in \Cnn$,
the product $A_1A_2=x_1y_1^*x_2y_2^*=(y_1^*x_2)x_1y_2^*$
is at most rank-one.
By a simple induction, the rank of arbitrary finite products of rank-one matrices 
remains at most one.

\subsection{The JSR of finite rank-one matrix family}
\label{rank-one result}
In this subsection, we will show that any finite set $\F$ of rank-one matrices possesses the finiteness property.
If $\rho(\F)=0$, then by (\ref{3ineq}) it holds  $\rho(A_i)=0$ for all $1\le i\le m$
and so the finiteness property is already true for $\F$. 
Thus we will only need to consider the case with $\rho(\F)>0$.

Recall that a general matrix family $\F$ is said to be irreducible, provided all the matrices in $\F$
have no common non-trivial invariant linear subspaces of $\IC^n$.
The following lemma indicates that $\F$ can be assumed to be  irreducible, since otherwise we could reduce $\F$ 
into several irreducible matrix families with smaller dimensions,
and then carry out the same proof arguments with each irreducible matrix family to draw the same conclusion.

\begin{Lemma}[\cite{Berger1992}]
\label{reducible}
For any finite matrix family $\F=\{A_1,A_2,\cdots,A_m\}\subset\IC^{n\times n}$, there exist
a nonsingular matrix $P\in\IC^{n\times n}$ and $r$ positive integers 
$\{n_1,n_2,\cdots,n_r\}$ with $n_1+n_2+\cdots+n_r=n$ such that
\[
 PA_iP^{-1}=
\bmat{A_i^{(1)}&0&\cdots&0\\
*&A_i^{(2)}&\cdots&0 \\
\vdots&\vdots&\ddots&\vdots \\
*&*&\cdots&A_i^{(r)}
}
\tn{for} i=1,2,\cdots,m,
\]
where $\F^{(j)}:=\{A_1^{(j)},A_2^{(j)},\cdots,A_m^{(j)}\}\subset \IC^{n_j\times n_j}$
is irreducible for $j=1,2,\cdots,r$, satisfying
\[
 \rho(\F)=\max_{1\le j\le r}\rho(\F^{(j)}).
\]
\end{Lemma}
Therefore, without loss of generality, we may always assume the matrix family $\F$ being irreducible. This leads to an important connection between the joint spectral radius and a special induced matrix norm, called extremal norm \cite{Jungers2009Book}. We further assume $\rho(\F)=1$ after normalizing $\F$ by dividing $\rho(\F)$ and also the irreducibility of $\F$ guarantees that the normalized $\F$ is non-defective, 
i.e., the semi-group of matrices generated by $\F$ is bounded,
and hence there exists an (invariant) extremal norm for $\F$ as described in the next lemma.
\begin{Lemma}[\cite{Wirth2002}]
\label{Bnorm}
 For any finite irreducible matrix family $\F$, 
 there exists a vector (Barabanov) norm $\|\cdot\|_B$ such that:
\begin{enumerate}
 \item[(1)] For all $v\in \IC^n$ and all $A\in\F$ it holds that
$
  \|Av\|_B\le \rho(\F) \|v\|_B,
$
\item[(2)] For all $v\in \IC^n$, there exists an $A\in\F$ such that
$
 \|Av\|_B=\rho(\F) \|v\|_B.
$
\end{enumerate}
In particular, the induced matrix norm $\|\cdot\|_B$ is an extremal norm satisfying
\[
  \max_{A\in\F}\|A\|_B=\rho(\F).
\]
\end{Lemma}
We are ready to prove the finiteness property of irreducible rank-one matrix family.
\begin{Theorem}
\label{rank1finite1}
Let $\F=\{A_i=x_iy_i^*:i=1,2,\cdots,m\}\subset\IC^{n\times n}$ be an irreducible rank-one matrix family.
Then $\F$ has the finiteness property and
the corresponding spectral maximizing product sequence of minimal length has distinct factors.
\end{Theorem}
\begin{proof}
We first normalize $\F$ such that $\rho(\F)=1$.
By Lemma \ref{Bnorm}, for any given $v\in \IC^n$ with $\|v\|_B=1$,
then for any $k\ge 1$ there exists a multi-index $(i_1,i_2,\cdots,i_k)$ such that
\begin{equation}
\label{bnorm-s}
  1=\|v\|_B=\|A_{i_1}v\|_B=\|A_{i_2}A_{i_1}v\|_B=\cdots=\|A_{i_k}\cdots A_{i_2}A_{i_1}v\|_B.
\end{equation}
By the pigeonhole principle, if $k\ge (m+1)$, then 
the multi-index $(i_1,i_2,\cdots,i_k)$ has at least one repeated index.
We define $s$ to be the maximum of those $k$'s such that 
the corresponding multi-index $(i_1,i_2,\cdots,i_k)$ satisfying (\ref{bnorm-s}) has no repetition.
It is obvious that $s\le m$.
Then, choosing $k=s+1$ in (\ref{bnorm-s}) gives $i_{s+1}=i_{j}$ for some unique $1\le j\le s$, that is,
 \[
 1=\|v\|_B=\cdots=\|A_{i_j}\cdots A_{i_1}v\|_B=\cdots=\|A_{i_{s+1}}A_{i_{s}}\cdots A_{i_j}\cdots A_{i_1}v\|_B.
\]
Since $A_{i_{s+1}}=A_{i_j}$ is rank-one, its range is one-dimensional and hence
\[
 A_{i_j}\cdots A_{i_1}v=\alpha z \tn{and} A_{i_{s+1}}A_{i_{s}}\cdots A_{i_j}\cdots A_{i_1}v=\beta z
\]
for some $0\ne\alpha\in\IC, 0\ne \beta\in\IC$, and $0\ne z\in\IC^n$
(we may choose $z=x_{i_j}$ here).
Then
\[
 \|\alpha z\|_B=\|A_{i_j}\cdots A_{i_1}v\|_B=1
=\|A_{i_{s+1}}A_{i_{s}}\cdots A_{i_j}\cdots A_{i_1}v\|_B=\|\beta z\|_B,
\]
which gives $|\alpha|=|\beta|$.
Finally, we obtain
\begin{align*}
 \beta z&=A_{i_{s+1}}A_{i_{s}}\cdots A_{i_{j+1}}(A_{i_j}\cdots A_{i_1}v)
=A_{i_{s+1}}A_{i_{s}}\cdots A_{i_{j+1}}(\alpha z)
\end{align*}
and hence
\[
 A_{i_{s+1}}A_{i_{s}}\cdots A_{i_{j+1}} z=\frac{\beta}{\alpha}z,
\]
where $\frac{\beta}{\alpha}$ is an eigenvalue of product $ A_{i_{s+1}}A_{i_{s}}\cdots A_{i_{j+1}}$.
Therefore, by Lemma \ref{Bnorm}, we get
$$1\ge \|A_{i_{s+1}}A_{i_{s}}\cdots A_{i_{j+1}}\|_B\ge 
\rho(A_{i_{s+1}}A_{i_{s}}\cdots A_{i_{j+1}})\ge |\frac{\beta}{\alpha}|=1,$$
which proves that $\F$ has the finiteness property with
\[
 \rho(\F)=1=\rho(A_{i_{s+1}}A_{i_{s}}\cdots A_{i_{j+1}})^{1/(s-j+1)},
\]
where $1\le (s-j+1)\le m$ and $i_{s+1}\ne i_s \ne \cdots \ne i_{j+1}$ by the choice of $s$.
\end{proof}

We remark here that Theorem \ref{rank1finite1} provides us an important structure of a spectral maximizing sequence, which will greatly improve the efficiency
of specially designed search algorithms. 
In particular, non-repeated index indicates that the lengths of all minimal spectral maximizing sequences
will not be longer than $m$. 
In fact, the possible minimal spectral maximizing sequence with longest length 
is $A_{i_1}A_{i_2}\cdots A_{i_m}$ with $i_s\ne i_t$ when $s\ne t$.  
In summary, an explicit formula for the JSR of any rank-one matrix family is given by 
$\F=\{A_1,A_2,\cdots,A_m\}\subset\IC^{n\times n}$ as
\begin{equation}
 \label{rank1formula}
\rho(\F)
=\max_{1\le k\le m}\left( \max_{A\in \F^{(*)}_k} \rho(A)^{1/k}\right),
\end{equation}
where $\F^{(*)}_k=\{A_{i_1}A_{i_2}\cdots A_{i_k}\in \F_k:\; i_s\ne i_t\quad \hbox{when}\; s\ne t\}$
denotes all possible  products in $\F_k$ with distinct factors.

Ahmadi and Parrilo recently show the same result by using the maximum cycle approach in graph theory \cite{Ahmadi2012}. 
In their work, they indicate that by using the well-known Karp algorithm (or its improved version) the computation of $\rho(\F)$ can be achieved in a polynomial computational time with $O(m^3+m^2n)$ complexity.

The following corollary extends our recent result appeared in \cite{Dai2011}. Morris has shown a broader version in his recent work \cite{Morris2011}.
\begin{Corollary} Let $\F=\{A_1, A_2, \dots, A_m\}\subset\IC^{n\times n}$ be irreducible with $A_j$ being rank-one for $2\le j\le m$, then $\F$ has the finiteness property. 
\end{Corollary}
\begin{proof} If $A_1$ is rank-one, then we are done by Theorem 3. Thus we assume the rank of $A_1$ is greater than one.
 
We first normalize $\F$ such that $\rho(\F)=1$
and then prove the conclusion by contradiction. Suppose the finiteness property does not hold. Then we have $\rho(A_j)< 1$ for $1\le j\le m$.

According to Lemma \ref{Bnorm},  for a given $v\in \IC^n$ with $\|v\|_B=1$,
then for any $k\ge 1$ there exists a sequence 
\[A_{i_k}\cdots A_{i_2}A_{i_1}\]
such that
\begin{equation}
\label{bnorm-s2}
  1=\|v\|_B=\|A_{i_1}v\|_B=\|A_{i_2}A_{i_1}v\|_B=\cdots=\|A_{i_k}\cdots A_{i_2}A_{i_1}v\|_B.
\end{equation}
When $k$ is large enough, the sequence $A_{i_k}\cdots A_{i_2}A_{i_1}$ has to be rank-one, since otherwise we have
\[\|A_{i_k}\cdots A_{i_2}A_{i_1}v\|_B=\|A_1^kv\|_B \leq \|A_1^k\|_B\to 0\]
as $k\to \infty$ due to $\rho(A_1)<1$.
Thus an index $i_j\in \{2, 3, \dots, m\}$ must appear at least twice as $k$ becomes sufficiently large. 
Let's consider the sequence 
\[A_{i_k}\cdots A_{i_{j+1}}A_{i_j}A_{i_{j-1}}\cdots A_{i_2}A_{i_1}\]
with $A_{i_k}=A_{i_j}$. 
Notice that the both finite products  $A_{i_k}\cdots A_{i_{j+1}}$ and $A_{i_j}A_{i_{j-1}}\cdots A_{i_2}A_{i_1}$ have the same (one-dimensional) range space. Hence there exist $z\in\IC^n, z\ne 0$, $\alpha\in\IC$ such that
\[
 A_{i_j}A_{i_j-1}\cdots A_{i_2}A_{i_1} v=\alpha z, \quad \|\alpha z\|_B=1
\]
and $\beta\in\IC$ such that
\[
A_{i_k}\cdots A_{i_{j+1}}(\alpha z)=\beta z, \quad \|\beta z\|_B=1.
\]
This implies that 
\[
A_{i_k}\cdots A_{i_{j+1}}z=\frac{\beta}{\alpha} z.
\]
Notice that $\|\alpha z\|_B=1=\|\beta z\|_B$ gives $|\alpha|=|\beta|$ since $\|z\|_B\ne 0$, we thus have
\[\rho(A_{i_k}\cdots A_{i_{j+1}})=1=\rho(\F)\]
which leads to a contradiction.
\end{proof}
Similar to the previous case, the JSR of $\F$ can be computed by
\begin{equation}
\label{rank1formula2}
\rho(\F)=\max\Big\{\rho^*,\max_{1\le i\le m}  \rho(A_i) \Big\}
\end{equation}
where
\[\rho^*=\max_{n_{i_j}\ge 0,  0<\ell\le m}\rho^{\frac{1}{n_1+\cdots+n_\ell+\ell}}(A_1^{n_1}A_{i_1}A_1^{n_2}A_{i_2}\cdots A_{n_\ell}^{n_\ell}A_{i_{\ell}})\]
with $2\le i_j\le m$ and $i_s\ne i_t$ when $s\ne t$. Again the Karp algorithm is still applicable for this case due to the existence of the maximum cycle according to our proof.

\subsection{Theoretical Examples}
In this subsection, we verify our foregoing results by some toy examples.
The formula (\ref{rank1formula}) provides a straightforward way to calculate the JSR for a rank-one matrix family.
The search of all possible products with distinct factors of length not exceeding $m$  
is sufficient to obtain the exact value of $\rho(\F)$.
However, most of current numerical approximation methods can only provide lower and upper bounds
for JSR with no indication whether the JSR has been achieved.
In particular, our formula (\ref{rank1formula}) is fully validated by the reported spectral maximizing sequences 
for any pair of rank-one $2\times 2$ sign-matrices in \cite{Cicone2010}.

\begin{example}[\cite{Cicone2010}]
 Consider the rank-one matrix pair
\[\F=\set{
A_1=\bmat{1& 1\\-1& -1},
A_2=\bmat{0& 1\\0& 1}
}.
\]
\end{example}
Apply the formula (\ref{rank1formula}) to obtain 
\[
\rho(\F)=\max_{1\le k\le 2} \max_{A\in \F^{(*)}_k} \rho(A)^{1/k}=\max\{\rho(A_1),\rho(A_2),\rho(A_1A_2)^{1/2}\}=\sqrt{2}.
\]
While in \cite{Cicone2010} this was solved by constructing an extremal real polytope norm.
\begin{example}[\cite{Guglielmi2005}]
Consider the rank-one matrix family
\[\F=\set{
A_1=\bmat{1& 1\\0& 0},
A_2=\bmat{0& 0\\1& 1},
A_{3}=\bmat{\frac{1}{2}& \frac{1}{2}\\ \frac{1}{2}& \frac{1}{2}},
A_{4}=\bmat{\frac{2}{3}& 0\\ \frac{-2}{3}& 0}
}.
\]
\end{example}
Using the formula (\ref{rank1formula}) to get 
\[
\rho(\F)=\max_{1\le k\le 4} \max_{A\in \F^{(*)}_k} \rho(A)^{1/k}=1.
\]
The same conclusion was derived in \cite{Guglielmi2005} by observing relations
among all matrices. But this approach is hard to be applied to general cases.
\begin{example}
 Consider the matrix family
\[\F=\set{
A_1=\bmat{1& 1&0\\0&1&1\\ 0&0&1},
A_2=\bmat{0& 0&1\\0&0&0\\ 0&0&0},
A_3=\bmat{0& 0&0\\0&0&0\\ 1&0&0}
}.
\]
\end{example}
A straightforward calculation based on formula (\ref{rank1formula2}) gives
\[
\rho(\F)=\rho(A_1^6A_3)^{1/7}=15^{1/7}\approx  1.472356700180347.
\]
\section{Rank-one approximation of JSR}
Although the JSR formula for rank-one matrix family is now available,
its applicability is highly restricted since rank-one matrix family 
rarely occurs in practice. 
Therefore, in following two subsections we develop an approximation approach to
expand its horizon of application. 
This method imitates the conventional definition of $\rho(\F)$ in terms of limit superior
and provides a new viewpoint on the approximation of JSR. 
The main idea results from the fact that the rank of any matrix products for a given set of matrices does not increase 
as the multiplication continues. This property provides us a feasible approach.

\subsection{General matrix family}
In this subsection, we will introduce a natural and insightful way of approximating $\rho(\F)$
by utilizing the previous results on rank-one matrix family.
The key idea is to perform the rank-one approximation of $\F_k$ successively as $k$ increases.
Let $A\in\IC^{n\times n}$, from its singular value decomposition (SVD) we have the following rank-one decomposition
\[
 A=\sum_{i=1}^n\sigma_i u_i v_i^*,
\]
where  $\sigma_1\ge \sigma_2\ge\cdots\ge\sigma_n\ge 0$ are the singular values, and
$u_i$ and $v_i$ are the $i$th left and right singular vector, respectively.
The customary best rank-one approximation of $A$ is trying to minimize $\|A-R\|_F$ 
over all rank-one matrices $R$, which is achieved by choosing $R=\sigma_1 u_1 v_1^*$.
For our approach, we will choose a special candidate, denoted by $\Rk(A)$, 
which maximizes the absolute value of its trace (or spectral radius), i.e.,
\[
 \Rk(A)\equiv\sigma_{i'} u_{i'} v_{i'}^*=\arg \max_{1\le i\le n} \left|\tr(\sigma_i u_i v_i^*)\right|
=\arg \max_{1\le i\le n} \left|\rho(\sigma_i u_i v_i^*)\right|
\]
as the rank-one approximation of $A$. Clearly $\sigma_{i'}\le \sigma_1=\|A\|_2$.
For the convenience of further discussion, 
we denote the element-wise rank-one approximation of
\[\F=\set{A_1,A_2,\ldots,A_m}\subset\IC^{n\times n}\]
by
\[\Rk(\F)=\set{\Rk(A_1),\Rk(A_2),\ldots,\Rk(A_m)},\]
where $\Rk(A_i)$ is the rank-one approximation of $A_i$ as defined above.
Notice that $\Rk(\F_k)$ is a finite rank-one matrix family,
thus $\rho(\Rk(\F_k))$  can be obtained by the  formula (\ref{rank1formula}).
Our next result needs the trace characterization of JSR by Chen and Zhou \cite{Chen2000} 
\begin{equation}
\label{jsrtrace}
\rho(\F)=\limsup_{k\to\infty}\max_{A\in\F_k}|\tr(A)|^{1/k}.
\end{equation}

\begin{Theorem}
\label{rankonesup}
For any finite matrix family $\F=\set{A_1,A_2,\ldots,A_m}\subset \Cnn$, there holds
 \bea
\label{rank1sup}
  \rho(\F)=\limsup_{k\to\infty}\rho(\Rk(\F_k))^{1/k}.
 \eea
\end{Theorem}
\begin{proof}
Let $\|\cdot\|_2$ denotes the spectral matrix norm. 
To derive our conclusion, we will prove (\ref{rank1sup}) by validating two inequalities as shown below.

Firstly, given any $k\ge 1$, for the finite rank-one matrix family $\Rk(\F_k)$, we have
\[
 \rho(\Rk(\F_k))=\max_{1\le l\le m^k}\max_{R\in \Rk(\F_k)_{l}} \rho(R)^{1/l},
\]
where $m^k$ is the cardinality of the set $\Rk(\F_k)$.
For any $R \in \Rk(\F_k)_{l}$, 
there exist $l$ matrices $B_j\in \F_k$ for  $1\le j\le l$ such that 
\[
 R=\Pi_{j=1}^{l} \Rk(B_j)=\Pi_{j=1}^{l} \sigma_j u_j v_j^* ,
\]
where $\Rk(B_j)=\sigma_j u_j v_j^*$ is the rank-one approximation of $B_j$ defined above.
By Cauchy inequalities $|v_j^*u_j|\le 1$ for  $1\le j\le l$ there holds 
$$\rho(R)\le \Pi_{j=1}^{l} \sigma_j \le \Pi_{j=1}^{l} \|B_j||_2\le
 \left(\max_{A\in\F_k} \|A||_2\right)^{l}$$
for any $R \in \Rk(\F_k)_{l}$.
Thus we have
\bean
 \rho(\Rk(\F_k))\le \max_{1\le l\le m^k} \left(\max_{A\in\F_k} \|A||_2\right)
= \max_{A\in \F_k} \|A\|_2,
\eean
which leads to
\bea
 \limsup_{k\to\infty}\rho(\Rk(\F_k))^{1/k}
&\le & \limsup_{k\to\infty}\left(\max_{A\in  \F_k}  \|A\|_2\right)^{1/k}=\rho(\F).
\label{ineq1}
\eea
Secondly,
for any $k\ge 1$, let 
$$B_{k}=\arg \max_{A\in\F_{k}}|\tr(A)|.$$
After expressing the SVD of $B_k$ as
$$
 B_{k}=\sum_{i=1}^{n}\hat\sigma_{i} \hat u_{i} \hat v_{i}^*,
$$
the linearity of trace operator gives
\bean
 |\tr(B_{k})|\le \sum_{i=1}^{{n}}|\tr(\hat\sigma_{i} \hat u_{i} \hat v_{i}^*)|
\le {n} \left(\max_{1\le i\le {n}} \left|\tr(\hat\sigma_{i} \hat u_{i} \hat v_{i}^*)\right|\right)
={n} \rho(\Rk(B_{k})).
\eean
Therefore, for $k\ge 1$ we have
\bean
\rho(\Rk(\F_{k}))^{1/{k}}\ge \rho(\Rk(B_{k}))^{1/{k}}\ge \left(n^{-1}|\tr(B_{k})|\right)^{1/{k}}= n^{-1/{k}}\left( \max_{A\in\F_{k}}|\tr(A)|\right)^{1/k},
\eean
which, in together with the equality (\ref{jsrtrace}), gives
\bea
\label{ineq2}
\limsup_{k\to\infty}\rho(\Rk(\F_k))^{1/k}&\ge&
\limsup_{k\to\infty} n^{-1/{k}}\left( \max_{A\in\F_{k}}|\tr(A)|\right)^{1/k}=\rho(\F).
\eea

Finally, by combining (\ref{ineq1}) and (\ref{ineq2}), the proof is completed.
\end{proof}
The above result reveals that the rank-one approximation of a matrix family will
approximate its JSR in the sense of limit superior.

In order to demonstrate the effectiveness of our approach in some important cases, let's first consider 
the following well-known $2\times 2$ irreducible matrix pair
\[\F=\set{
A_1=\bmat{1& 1\\0& 1},
A_2=b\bmat{1& 0\\1& 1}
}
\]
with $b>0$, which was recently employed to disprove the finiteness conjecture for some $b\in (0, 1)$ \cite{Blondel2003}.
The authors in \cite{Guglielmi2008} show 
by constructing an exact real polytope extremal norm along with computational investigation that
$\rho(\F)=\sigma_1 \sqrt{b}$ when $b\in[\frac{4}{5},1]$ and the minimal spectral maximizing sequence is $A_1A_2$. 

Note that the SVD of $A_1$ (calculated by Mathematica 8) is given by
\bean
 A_1&=&\bmat{u_1& u_2} 
\bmat{
\sigma_1 & 0 \\
 0 &\sigma_2
}
\bmat{v_1 &v_2}^*
\\
&=&\bmat{
 \frac{1+\sqrt{5}}{\sqrt{2 \left(5+\sqrt{5}\right)}} & \frac{1-\sqrt{5}}{\sqrt{10-2 \sqrt{5}}} \\
 \sqrt{\frac{2}{5+\sqrt{5}}} & \sqrt{\frac{1}{10} \left(5+\sqrt{5}\right)}
}
\bmat{
 \frac{\sqrt{5}+1}{2} & 0 \\
 0 & \frac{\sqrt{5}-1}{2}
}
\bmat{
 \frac{-1+\sqrt{5}}{\sqrt{10-2 \sqrt{5}}} & -\frac{1+\sqrt{5}}{\sqrt{2 \left(5+\sqrt{5}\right)}} \\
 \sqrt{\frac{1}{10} \left(5+\sqrt{5}\right)} & \sqrt{\frac{2}{5+\sqrt{5}}}
}^*,
\eean
and thus the SVD of $A_2$ is
\[
 A_2=bA_1^*=\bmat{v_1& v_2} 
\bmat{
b\sigma_1 & 0 \\
 0 &b\sigma_2
}
\bmat{u_1 &u_2}^*.
\]
Direct calculation yields $u_1^* u_1=v_1^* v_1=1$ and $v_1^*u_1=u_1^*v_1=\frac{2}{\sqrt{5}}$.
Hence, the rank-one approximation of $\F$ based on SVD  is 
\[\Rk(\F_1)=\Rk(\F)=\set{
\sigma_1 u_1 v_1^*,
b\sigma_1 v_1 u_1^*
}.
\]
According to (\ref{rank1formula}), if $b\in [\frac{4}{5},1]$ then there holds
\bean
\rho(\Rk(\F_1))&=&\max\set{\sigma_1 v_1^*u_1,b\sigma_1 u_1^*v_1,\sigma_1\sqrt{b}\sqrt{u_1^* u_1\cdot v_1^* v_1}}\\
&=&\sigma_1 \cdot \max\set{ \frac{2}{\sqrt{5}},b \frac{2}{\sqrt{5}}, \sqrt{b}}=\sigma_1 \sqrt{b}=\rho(\F).
\eean
Here only one step of rank-one approximation gives $\rho(\F)$.

\subsection{Nonnegative matrix family}
In numerical implementation, it would be more favorable to have a limit rather than a limit superior
in (\ref{rank1sup}) because the former can be evidently observed within sufficient steps of approximations. In this subsection, we further develop the limit property for nonnegative matrix family.

In this paper by the notation $A\ge 0$ and $A>0$ we mean the matrix $A$ 
is nonnegative and positive in {\it entry-wise} sense, respectively.
If the considering matrix family $\F=\set{A_1,A_2,\ldots,A_m}\subset \IR^{n\times n}$ is nonnegative, i.e., $A_i\ge 0$,
there is an elegant limit expression of JSR established by Blondel and Nesterov \cite{Blondel2005}
\bea
\label{jsrkron}
\rho(\F)=\lim_{k\to\infty}\rho^{1/k}(A_1^{\otimes k}+\ldots +A_m^{\otimes k}),
\eea
where $A_i^{\otimes k}$ represents the $k$-th Kronecker power of $A_i$.  This expression will play an important role in our following approach. For square matrices of the same size,  the following properties for Kronecker product can be found in a standard matrix analysis textbook:
\begin{itemize}
\item[{\rm (i)}] $(A_i\otimes{A_j})(A_s\otimes{A_t})=(A_iA_s)\otimes{(A_jA_t)};$
\item[{\rm (ii)}] ${\rm tr}\left(A_{i_1}A_{i_2}\cdots A_{i_{\ell}}\right)^{\otimes k}={\rm tr}^{k}\left(A_{i_1}A_{i_2}\cdots A_{i_{\ell}}\right)$ for any positive integer $k$;
\item[\rm(iii)]  $\left(A^{\otimes k}\right)^{\ell}=\left(A^{\ell}\right)^{\otimes k}$ for any positive integers $k, \ell$.
\end{itemize}

Recall that a square matrix $A\ge 0$ is called to be primitive if $A^{\ell}>0$ for some integer $\ell\ge 1$. 
It is easy to see that if $A$ is primitive, then $A^{\otimes k}$ is also primitive according to (iii).
An important property for a  primitive matrix is that its spectral radius can be expressed as
\bea
\label{primtrace}
\rho(A)=\lim_{k\to \infty}\tr^{1/k}(A^k),
\eea
instead of limit superior in general cases.
\begin{Lemma} 
\label{primlem}
Let $\F=\{A_1, \cdots, A_m\}\subset \IR^{n\times n}$ be a family of nonnegative matrices.
If there exists an integer $j\ge 1$ such that $A_{i_1}A_{i_2}\cdots A_{i_j}\in\F_j$ is primitive, 
then for any $k\ge 1$ the matrix
$
A_1^{\otimes k}+\ldots +A_m^{\otimes k}
$
is primitive.
\end{Lemma}

\begin{proof} Since $A_{i_1}A_{i_2}\cdots A_{i_j}$ is primitive, there exists a positive integer $\ell$ such that
\[ \left(A_{i_1}A_{i_2}\cdots A_{i_j}\right)^{\ell}>0.\]
The conclusion follows from the observation
\bean
\left[\left(A_1^{\otimes k}+\ldots +A_m^{\otimes k}\right)^j\right]^{\ell}&=&
\left[\sum_{1\le i_1, i_2, \cdots, i_j\le m} (A_{i_1}A_{i_2}\cdots A_{i_j})^{\otimes k}\right]^{\ell}\\
&\ge&\left[(A_{i_1}A_{i_2}\cdots A_{i_j})^{\otimes k}\right]^{\ell}=\left[(A_{i_1}A_{i_2}\cdots A_{i_j})^{\ell}\right]^{\otimes k}>0.
\eean
\end{proof}

The following lemma generalizes a recent result given by Xu \cite{Xu2010}.
\begin{Lemma} 
Let $\F=\{A_1, \cdots, A_m\}\subset \IR^{n\times n}$ be a family of nonnegative matrices. 
If there exists an integer $j\ge 1$ such that $A_{i_1}A_{i_2}\cdots A_{i_j}\in\F_j$ is primitive,
then we have
\[
\rho(\F)=\lim_{l\to \infty}\max_{A\in \F_l}{\tr^{1/l}(A)}.
\]
\end{Lemma}
\begin{proof}
 For any integers $k, ~l \ge 1$, there holds
\bean
{\rm tr}\left(A_1^{\otimes k}+\ldots +A_m^{\otimes k}\right)^l
 &=& {\rm tr}\left[\sum_{1 \le i_1, \ldots, i_l \le m}\left(A_{i_1}\cdots A_{i_l}\right)^{\otimes k}\right]\\
 &=& \sum_{1 \le i_1, \ldots, i_l \le m}{\rm tr}\left(A_{i_1}\cdots A_{i_l}\right)^{\otimes k}\\
 &=& \sum_{1 \le i_1, \ldots, i_l \le m}{\rm tr}^{k}\left(A_{i_1}\cdots A_{i_l}\right),
\eean
i.e. 
\[
{\rm tr}\left(A_1^{\otimes k}+\ldots +A_m^{\otimes k}\right)^l=\sum_{A \in \F_l}{\rm tr}^{k}(A),
\]
where the cardinality of $\F_l$ is $m^l$. This leads to
\[
\left[{\rm tr}^{1/l}\left(A_1^{\otimes k}+\ldots +A_m^{\otimes k}\right)^l\right]^{1/k}\le m^{1/k}\left(\max_{A\in \F_l}{\rm tr}^{1/l}(A)\right)
\]
By assuming that $A_{i_1}A_{i_2}\cdots A_{i_j}\in\F_j$ is primitive for some $j\ge 1$,
we know from Lemma \ref{primlem} that $A_1^{\otimes k}+\ldots +A_m^{\otimes k}$ is primitive for all positive integer $k$. For a fixed $k\ge 1$, by taking the limit inferior over $l$ on both sides and noting (\ref{primtrace}), we have
\[
\rho^{1/k}(A_1^{\otimes k}+\ldots +A_m^{\otimes k})\le m^{1/k} \liminf_{l\to \infty}\max_{A\in \F_l}{\rm tr}^{1/l}(A).
\]
Now by letting $k\to \infty$ and utilizing (\ref{jsrkron}) we obtain
\[
\rho(\F)\le \liminf_{l\to \infty}\max_{A\in \F_l}{\rm tr}^{1/l}(A).
\]
By combining this with the known equality
\[
\rho(\F)=\limsup_{l\to \infty}\max_{A\in \F_l}{\rm tr}^{1/l}(A),
\]
we derive
\[
\limsup_{l\to \infty}\max_{A\in \F_l}{\rm tr}^{1/l}(A)=\rho(\F)\le \liminf_{l\to \infty}\max_{A\in \F_l}{\rm tr}^{1/l}(A)
\]
which leads to the conclusion.
\end{proof}
\begin{Corollary}
\label{rankonelim}
 Let $\F=\{A_1, \cdots, A_m\}\subset \IR^{n\times n}$ be a family of nonnegative matrices. 
If there exists an integer $j\ge 1$ such that $A_{i_1}A_{i_2}\cdots A_{i_j}\in\F_j$ is primitive, 
then there holds
 \bea
\label{rank1lim}
\rho(\F)=\lim_{k\to\infty}\rho(\Rk(\F_k))^{1/k}.
\eea
\end{Corollary}
\begin{proof}
Similar to the proof of Theorem \ref{rankonesup}, by repeating the argument in the first part, it follows the inequality
\bea
 \limsup_{k\to\infty}\rho(\Rk(\F_k))^{1/k}&\le & \rho(\F).
\label{ineq1cor}
\eea
While in the second part, we have
\bean
\rho(\Rk(\F_{k}))^{1/{k}} &\ge & n^{-1/{k}}\left( \max_{A\in\F_{k}}|\tr(A)|\right)^{1/k}
\eean
for all $k\ge 1$,
taking the limit inferior on both sides and employing Lemma \ref{primtrace} gives
\bea
\label{ineq2cor}
\liminf_{k\to\infty}\rho(\Rk(\F_k))^{1/k}&\ge&
\lim_{k\to\infty} n^{-1/{k}}\left( \max_{A\in\F_{k}}|\tr(A)|\right)^{1/k}=\rho(\F).
\eea
According to (\ref{ineq1cor}) and (\ref{ineq2cor}), we thus have
\[
\limsup_{k\to\infty}\rho(\Rk(\F_k))^{1/k}\le \rho(\F)\le \liminf_{k\to\infty}\rho(\Rk(\F_k))^{1/k},
\]
which implies the conclusion.
\end{proof}

The primitive condition in Corollary \ref{rankonelim} is verifiable in numerical computations since a nonnegative matrix $A\in\Rnn$ is primitive if and only if $A^{n^2-2n+2}$ is positive \cite{Horn1990}.
Moreover, such an approximation to JSR tends to be observed earlier than the limit superior as $k$ increases,
which will be illustrated in the following numerical examples.




\section{Numerical examples}

To demonstrate the effectiveness of our proposed approach,
we next present numerical simulations on several important examples from current literature.  
All experiments are performed on MATLAB 7.10 with a machine precision of $10^{-16}$.
For $k\ge 1$, we denote
$$\hat{\rho}_k(\F)=\max_{A\in\F_k}\|A\|_2^{1/k},\quad
 \bar{\rho}_k(\F)=\max_{A\in\F_k}\rho(A)^{1/k}, \quad \mbox{and} \quad
\tilde{\rho}_k(\F)=\left[\rho(\Rk(\F_k))\right]^{1/k}.$$

\subsection{Generalized partition function}
In number theory, a long-lasting problem is to estimate 
the asymptotic growth of the generalized partition function
$f_{p,c}(t)$ defined as the total number of different $p$-adic expansions
 $t=\sum_{j=0}^{\infty} c_j p^j$ with $c_j\in\{0,1,\ldots,c-1\}$.
It has been shown that for given positive integers $p$ and $c$ there exist
positive constants $C$ and $\gamma$ such that
$
 f_{p,c}(t)\ge C t^{\gamma}
$
as $t\to\infty$.
Moreover, there is a procedure \cite{Protasov2000} to construct a family of binary matrices $\F$ dependent on $p$ and $c$
with the relation $\rho(\F)=p^{\gamma}$.

In \cite{Protasov2010}, a conic programming approach was proposed to 
approximate the JSR of this matrix family $\F$ of dimension $7\times 7$ with $(p,c)=(3,14)$.
In this case the matrix family is given by
\bean
\F&=&\left\{
A_1=\left[
\begin{array}{ccccccc}
 1 & 1 & 1 & 1 & 1 & 0 & 0 \\
 0 & 1 & 1 & 1 & 1 & 0 & 0 \\
 0 & 1 & 1 & 1 & 1 & 1 & 0 \\
 0 & 1 & 1 & 1 & 1 & 1 & 0 \\
 0 & 0 & 1 & 1 & 1 & 1 & 0 \\
 0 & 0 & 1 & 1 & 1 & 1 & 1 \\
 0 & 0 & 1 & 1 & 1 & 1 & 1
\end{array}
\right],
A_2=\left[
\begin{array}{ccccccc}
 1 & 1 & 1 & 1 & 1 & 0 & 0 \\
 1 & 1 & 1 & 1 & 1 & 0 & 0 \\
 0 & 1 & 1 & 1 & 1 & 0 & 0 \\
 0 & 1 & 1 & 1 & 1 & 1 & 0 \\
 0 & 1 & 1 & 1 & 1 & 1 & 0 \\
 0 & 0 & 1 & 1 & 1 & 1 & 0 \\
 0 & 0 & 1 & 1 & 1 & 1 & 1
\end{array}
\right]\right.,\\
&&\left.A_3=\left[
\begin{array}{ccccccc}
 1 & 1 & 1 & 1 & 0 & 0 & 0 \\
 1 & 1 & 1 & 1 & 1 & 0 & 0 \\
 1 & 1 & 1 & 1 & 1 & 0 & 0 \\
 0 & 1 & 1 & 1 & 1 & 0 & 0 \\
 0 & 1 & 1 & 1 & 1 & 1 & 0 \\
 0 & 1 & 1 & 1 & 1 & 1 & 0 \\
 0 & 0 & 1 & 1 & 1 & 1 & 0
\end{array}
\right]
\right\}.
\eean
According to  \cite{Protasov2010} the estimated interval of $\rho(\F)$ is given by
$[4.72,4.8]$, where the lower bound is attained by using the sequence $A_1A_2$ and
the upper bound is searched through the conic algorithm among all possible matrix products within length $k\le 9$. 
There is no indication what the value of $\rho(\F)$ in their approach is.  

 \begin{table}[!h]
\centering
\caption{The values of $\bar{\rho}_k(\F),\hat{\rho}_k(\F),\tilde{\rho}_k(\F)$ with respect to $k$, $(p,c)=(3,14)$.}
\begin{tabular}{|c|l|l|l|c|}
\hline
$k$&$\hat{\rho}_k(\F)$&$\bar{\rho}_k(\F)$&$\tilde{\rho}_k(\F)$&$|\tilde{\rho}_k(\F)-\rho^{1/2}(A_1A_2)|$\\
\hline
1&	5.262878 &	4.6690790883&	 4.7915415825& 	 6.95e-02 \\ 
2&	5.046134 &	4.7220451340&	 4.7208642368& 	 1.18e-03 \\ 
3&	4.936157 &	4.7122439907&	 4.7216905518& 	 3.55e-04 \\ 
4&	4.881518 &	4.7220451340&	 4.7220575153& 	 1.24e-05 \\ 
5&	4.849140 &	4.7164125255&	 4.7220470073& 	 1.87e-06 \\ 
6&	4.827731 &	4.7220451340&	 4.7220461006& 	 9.67e-07 \\ 
7&	4.812488 &	4.7180343424&	 4.7220452529& 	 1.19e-07 \\ 
8&	4.801089 &	4.7220451340&	 4.7220451879& 	 5.39e-08 \\ 
 \hline
\end{tabular}
\label{TBrankone7by7}
 \end{table} 
\begin{figure}[!h]
\caption{
The values of $\bar{\rho}_k(\F),\hat{\rho}_k(\F),\tilde{\rho}_k(\F)$ with respect to $k$, $(p,c)=(3,14)$.}
\centering{\epsfysize=3.0truein \leavevmode \epsffile{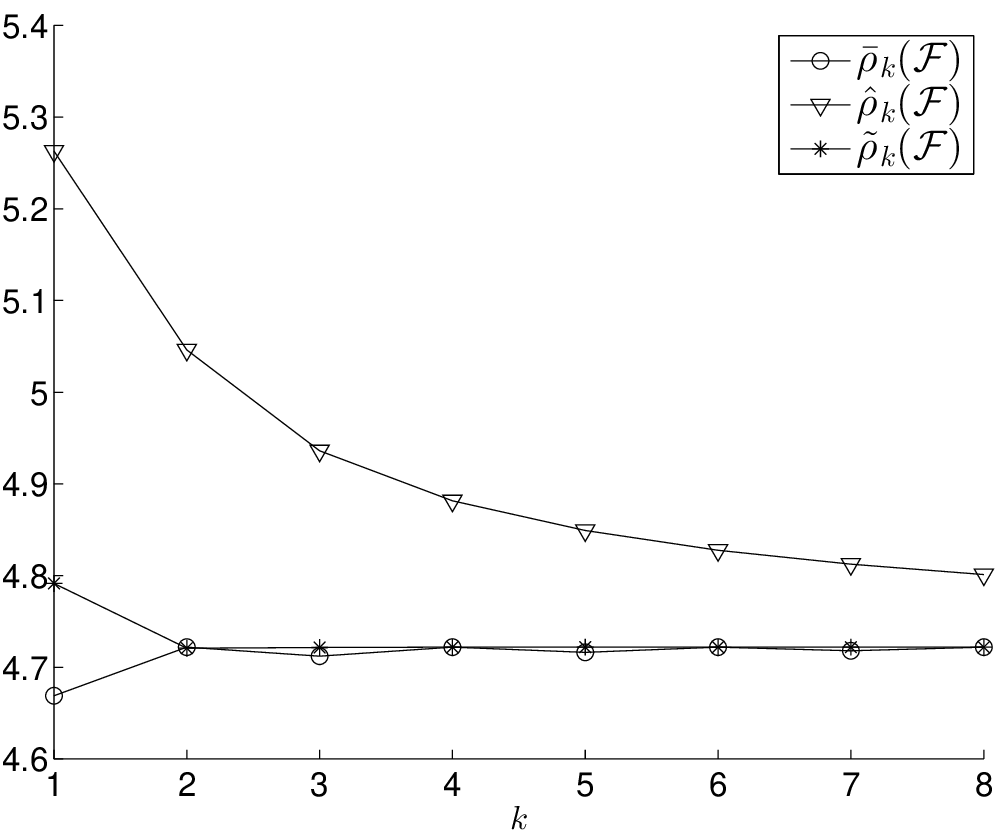}}
\label{rankone7by7}
\end{figure}

 
The first step rank-one approximation $\Rk(\F)$ by formula (\ref{rank1formula}) gives
\[
\rho(\Rk(\F))\approx 4.7915415825,
\]
which already falls into the interval $[4.72,4.8]$.  The numerical values of $\hat{\rho}_k(\F)$,
 $\bar{\rho}_k(\F)$, and $\tilde{\rho}_k(\F)$ for $k\ge 1$
are plotted and shown in Figure \ref{rankone7by7} and Table \ref{TBrankone7by7}, respectively.
Notice that the product $A_1A_3A_1$ is positive. 
According to Corollary \ref{rankonelim}, we know there holds $\lim_{k\to \infty}\tilde{\rho}_k(\F)=\rho(\F)$. 
Therefore, we also compute the absolute difference between $\tilde{\rho}_k(\F)$ and $\rho^{1/2}(A_1A_2)$,
 and the numerical results presented in Table 1 imply  that $A_1A_2$ is  very likely the spectral maximizing sequence of $\rho(\F)$.


\subsection{Asymptotics of Overlap-free words}
A word on the binary alphabet $\{a,b\}$ is said to be overlap-free if it has no sub-words (or factors) of the form $xwxwx$,
where $x\in\{a,b\}$ and $w$ could be a word or empty.
For instance, the word $baabaa$ is overlap-free, but $baabaab$ is not.
The asymptotic growth of the number $t_l$ of binary overlap-free words of length $l$
could be expressed in terms of the JSR of a matrix pair $\F$ as \cite{Jungers2009} 
\[
 \limsup_{l\to \infty}\frac{\ln t_l}{\ln l}=\log_2 \rho(\F),
\]
where $\F=\{A_1,A_2\}\subset \{0,1,2,4\}^{20\times 20}$ is given  by
\[
A_1=\bmat{C& \bm{0}\\D& B}\qquad \textnormal{and}\qquad
A_2=\bmat{D& B\\\bm{0}& C}
\]
with sub-matrices
\[
B=\left[\begin{array}{cccccccccc} 
0 & 0 & 0 & 0 & 0 & 0 & 0 & 1 & 2 & 1\\ 
0 & 0 & 0 & 0 & 0 & 0 & 0 & 0 & 0 & 0\\
 0 & 0 & 0 & 0 & 0 & 1 & 1 & 0 & 0 & 0\\ 
0 & 0 & 1 & 1 & 0 & 0 & 0 & 0 & 0 & 0\\ 
0 & 0 & 0 & 0 & 0 & 0 & 0 & 0 & 0 & 0\\ 
0 & 0 & 0 & 0 & 0 & 0 & 0 & 0 & 0 & 0\\ 
0 & 0 & 0 & 0 & 0 & 0 & 0 & 0 & 0 & 0\\
 0 & 0 & 0 & 0 & 1 & 0 & 0 & 0 & 0 & 0\\ 
0 & 1 & 0 & 0 & 0 & 0 & 0 & 0 & 0 & 0\\
 1 & 0 & 0 & 0 & 0 & 0 & 0 & 0 & 0 & 0 
\end{array}\right],
\]
\[
C=\left[\begin{array}{cccccccccc} 
0 & 0 & 0 & 0 & 0 & 0 & 0 & 2 & 4 & 2\\ 
0 & 0 & 1 & 1 & 0 & 1 & 1 & 0 & 0 & 0\\ 
0 & 0 & 0 & 0 & 0 & 1 & 1 & 1 & 1 & 0\\ 
0 & 0 & 1 & 1 & 0 & 0 & 0 & 0 & 0 & 0\\ 
0 & 0 & 0 & 0 & 0 & 0 & 0 & 0 & 0 & 0\\ 
0 & 1 & 0 & 0 & 1 & 0 & 0 & 0 & 0 & 0\\ 
1 & 1 & 0 & 0 & 0 & 0 & 0 & 0 & 0 & 0\\ 
0 & 0 & 0 & 0 & 0 & 2 & 0 & 0 & 0 & 0\\ 
0 & 0 & 1 & 0 & 0 & 0 & 0 & 0 & 0 & 0\\ 
0 & 0 & 0 & 0 & 0 & 0 & 0 & 0 & 0 & 0 
\end{array}\right],
\]
and \[
D=\left[\begin{array}{cccccccccc}
 0 & 0 & 0 & 0 & 0 & 0 & 0 & 1 & 2 & 1\\
 0 & 0 & 1 & 1 & 0 & 1 & 1 & 0 & 0 & 0\\ 
0 & 0 & 0 & 0 & 0 & 0 & 0 & 1 & 1 & 0\\
 0 & 0 & 0 & 0 & 0 & 0 & 0 & 0 & 0 & 0\\ 
1 & 2 & 0 & 0 & 1 & 0 & 0 & 0 & 0 & 0\\ 
0 & 0 & 1 & 0 & 0 & 1 & 0 & 0 & 0 & 0\\ 
0 & 0 & 0 & 0 & 0 & 0 & 0 & 0 & 0 & 0\\ 
0 & 0 & 0 & 0 & 0 & 0 & 0 & 1 & 0 & 0\\
 0 & 0 & 0 & 0 & 0 & 0 & 0 & 0 & 0 & 0\\ 
0 & 0 & 0 & 0 & 0 & 0 & 0 & 0 & 0 & 0 
\end{array}\right].
\]
This problem was firstly considered in \cite{Jungers2009} and then in \cite{Protasov2010},
where both ellipsoidal norm approximation and conic programming approach produce
the same bound, i.e., $\rho(\F)\in[2.5179,2.5186]$, among all matrix products within the length $k\le 14$.
In particular, $\rho(A_1A_2)^{1/2}\approx 2.5179$ gives the lower bound.

Moreover, the authors in \cite{Jungers2009} conjectured the sequence $A_1A_2$ is the spectral maximizing sequence. 
The numerical values of $\hat{\rho}_k(\F)$,  $\bar{\rho}_k(\F)$, and $\tilde{\rho}_k(\F)$
are plotted and presented in Figure \ref{rankone20by20} and Table \ref{TBrankone20by20}, respectively.
By observing that $$|\tilde{\rho}_k(\F)-\rho(A_1A_2)^{1/2}|\le 5\times 10^{-4}$$ when $k\ge 9$, 
our numerical results thus agree with their conjecture.

\begin{figure}[!h]
\caption{
The values of $\bar{\rho}_k(\F),\hat{\rho}_k(\F),\tilde{\rho}_k(\F)$ with respect to $k$.}
\centering{\epsfysize=3.0truein \leavevmode \epsffile{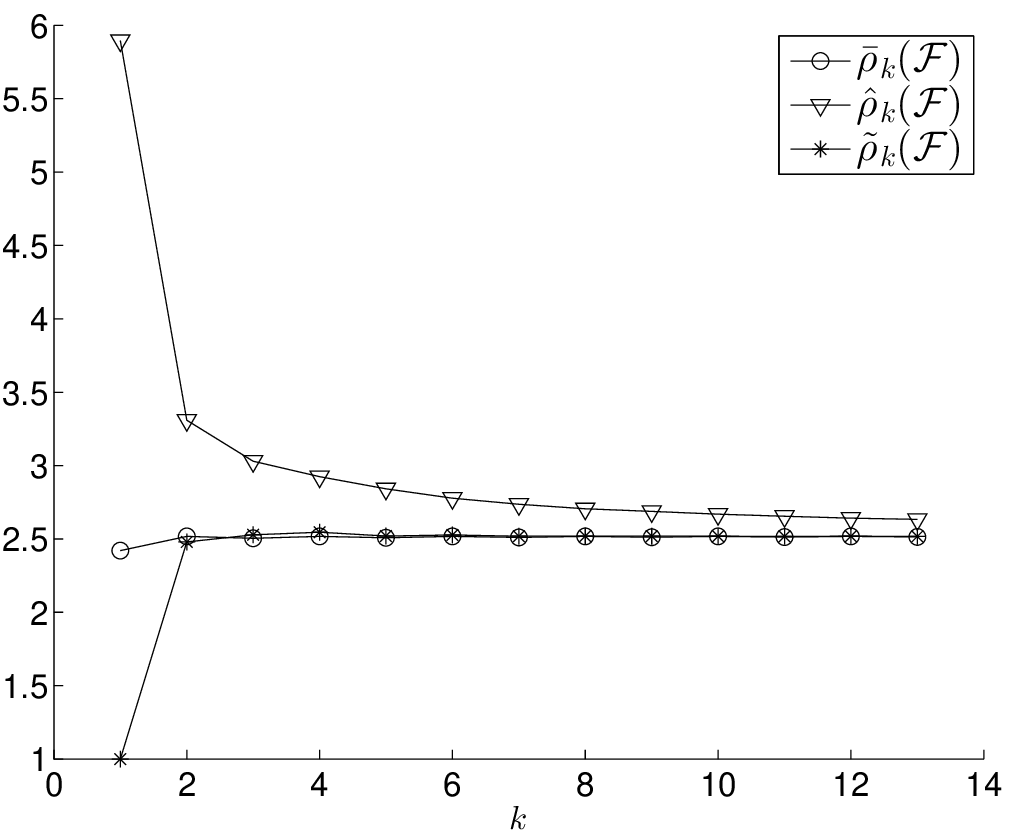}}
\label{rankone20by20}
\end{figure}

 \begin{table}[!h]
\centering
\caption{The values of $\bar{\rho}_k(\F),\hat{\rho}_k(\F),\tilde{\rho}_k(\F)$ with respect to $k$.}
\begin{tabular}{|c|l|l|l|c|}
\hline
$k$&$\hat{\rho}_k(\F)$&$\bar{\rho}_k(\F)$&$\tilde{\rho}_k(\F)$&$|\tilde{\rho}_k(\F)-\rho(A_1A_2)^{1/2}|$\\
\hline
1&	5.896964 &	2.4206250653&	 1.0000000000& 	 1.52e+00 \\ 
2&	3.309093 &	2.5179340409&	 2.4799585961& 	 3.80e-02 \\ 
3&	3.029307 &	2.5048603453&	 2.5279522425& 	 1.00e-02 \\ 
4&	2.924657 &	2.5179340409&	 2.5459319895& 	 2.80e-02 \\ 
5&	2.841023 &	2.5080155383&	 2.5201385520& 	 2.20e-03 \\ 
6&	2.778162 &	2.5179340409&	 2.5268682549& 	 8.93e-03 \\ 
7&	2.736156 &	2.5099337275&	 2.5190358732& 	 1.10e-03 \\ 
8&	2.705763 &	2.5179340409&	 2.5199752844& 	 2.04e-03 \\ 
9&	2.687999 &	2.5118420373&	 2.5180910647& 	 1.57e-04 \\ 
10&	2.669268 &	2.5179340409&	 2.5184122994& 	 4.78e-04 \\ 
11&	2.654756 &	2.5129654473&	 2.5179476144& 	 1.36e-05 \\ 
12&	2.642173 &	2.5179340409&	 2.5180554550& 	 1.21e-04 \\ 
13&	2.632798 &	2.5137397302&	 2.5179399051& 	 5.86e-06 \\ 
 \hline
\end{tabular}
\label{TBrankone20by20}
 \end{table}

\subsection{An example with oscillated approximation}
Consider the matrix pair \cite{Morris2010}
\[\F=\set{
A_1=\bmat{0& 1\\1& 0},
A_2=\bmat{1/2& 0\\0& 1/2}
}.
\]
The numerical values of $\hat{\rho}_k(\F)$,
 $\bar{\rho}_k(\F)$, and $\tilde{\rho}_k(\F)$
are plotted and reported in Figure \ref{rankoneworst} and Table \ref{TBrankoneworst}, respectively.
Although this matrix family does not satisfy the assumption of Corollary \ref{rankonelim}, one can
easily show that $\rho(\F)=\lim_{k\to \infty} \tilde{\rho}_k(\F)$. This example illustrates that the assumption of Corollary \ref{rankonelim} is sufficient but not necessary. 
\begin{figure}[!h]
\caption{
The values of $\bar{\rho}_k(\F),\hat{\rho}_k(\F),\tilde{\rho}_k(\F)$ with respect to $k$.}
\centering{\epsfysize=3.0truein \leavevmode \epsffile{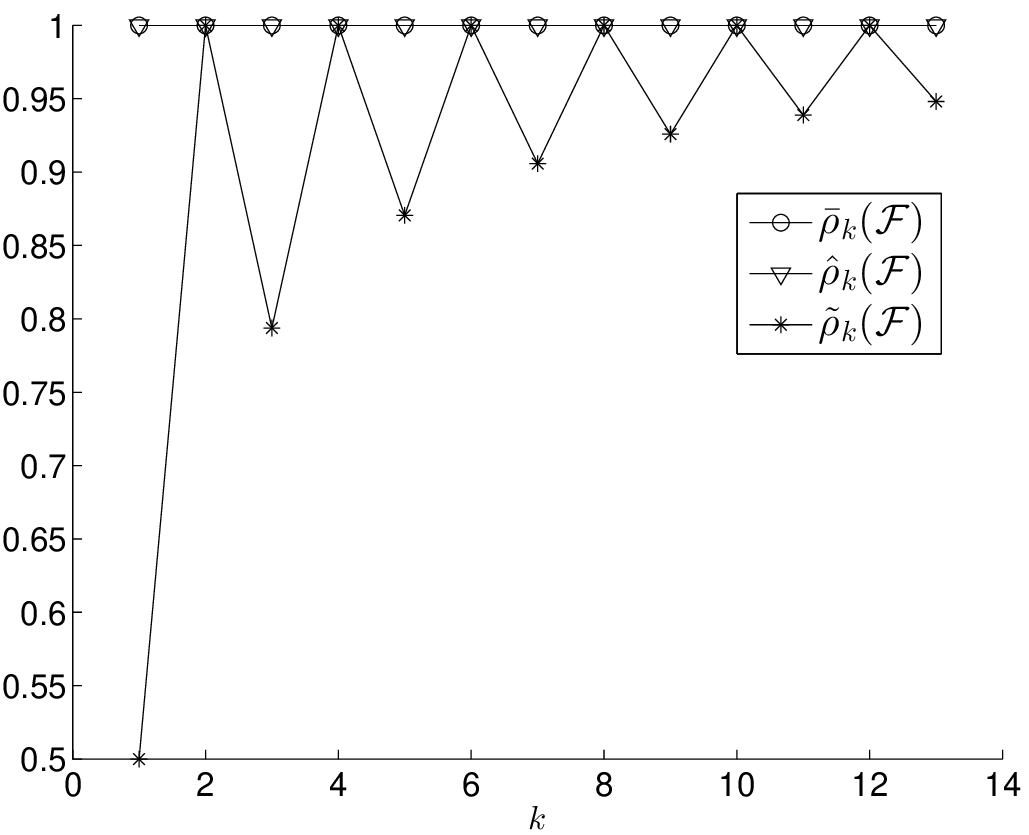}}
\label{rankoneworst}
\end{figure}

 \begin{table}[!h]
\centering
\caption{The values of $\bar{\rho}_k(\F),\hat{\rho}_k(\F),\tilde{\rho}_k(\F)$ with respect to $k$.}
\begin{tabular}{|c|l|l|l|c|}
\hline
$k$&$\hat{\rho}_k(\F)$&$\bar{\rho}_k(\F)$&$\tilde{\rho}_k(\F)$&$|\tilde{\rho}_k(\F)-\rho(\F)|$\\
\hline
1&	1.000000 &	1.0000000000&	 0.5000000000& 	 5.00e-01 \\ 
2&	1.000000 &	1.0000000000&	 1.0000000000& 	 0 \\ 
3&	1.000000 &	1.0000000000&	 0.7937005260& 	 2.06e-01 \\ 
4&	1.000000 &	1.0000000000&	 1.0000000000& 	 0 \\ 
5&	1.000000 &	1.0000000000&	 0.8705505633& 	 1.29e-01 \\ 
6&	1.000000 &	1.0000000000&	 1.0000000000& 	 0 \\ 
7&	1.000000 &	1.0000000000&	 0.9057236643& 	 9.43e-02 \\ 
8&	1.000000 &	1.0000000000&	 1.0000000000& 	 0 \\ 
9&	1.000000 &	1.0000000000&	 0.9258747123& 	 7.41e-02 \\ 
10&	1.000000 &	1.0000000000&	 1.0000000000& 	 0 \\ 
11&	1.000000 &	1.0000000000&	 0.9389309107& 	 6.11e-02 \\ 
12&	1.000000 &	1.0000000000&	 1.0000000000& 	 0 \\ 
13&	1.000000 &	1.0000000000&	 0.9480775143& 	 5.19e-02 \\ 
 \hline
\end{tabular}
\label{TBrankoneworst}
 \end{table} 
\subsection{Matrix pair with a rank-one member}
Consider the matrix pair
\[\F=\set{
A_1=\bmat{1& \frac{1}{\sqrt{7}}\\0& 1},
A_2=\bmat{1& -1\\1& -1}
}.
\]
Based on formula (\ref{rank1formula2}), it is easy to show that
\[\rho(\F)=\max_{\ell\ge 1} \left(\frac{\ell}{\sqrt{7}}\right)^{1\over \ell+1}
=\left(\frac{8}{\sqrt{7}}\right)^{1/ 9}\approx 1.130819895422034
\]
with the spectral maximizing sequence $A_1^{8} A_2$.
The numerical values of $\hat{\rho}_k(\F)$,
 $\bar{\rho}_k(\F)$, and $\tilde{\rho}_k(\F)$ are plotted in Figure \ref{rankone2by2}
and reported in Table \ref{TBrankone2by2}, respectively.
One can see that for small $k$,  $\tilde{\rho}_k(\F)$
approaches $\rho(\F)$ much better than $\hat{\rho}_k(\F)$ and $\bar{\rho}_k(\F)$. 
Moreover, the approximation $\tilde{\rho}_9(\F)$ does provide the exact value of $\rho(\F)$,
because $A_1^{8} A_2$ is rank-one.
\begin{figure}[!h]
\caption{
The values of $\bar{\rho}_k(\F),\hat{\rho}_k(\F),\tilde{\rho}_k(\F)$ with respect to $k$.}
\centering{\epsfysize=3.0truein \leavevmode \epsffile{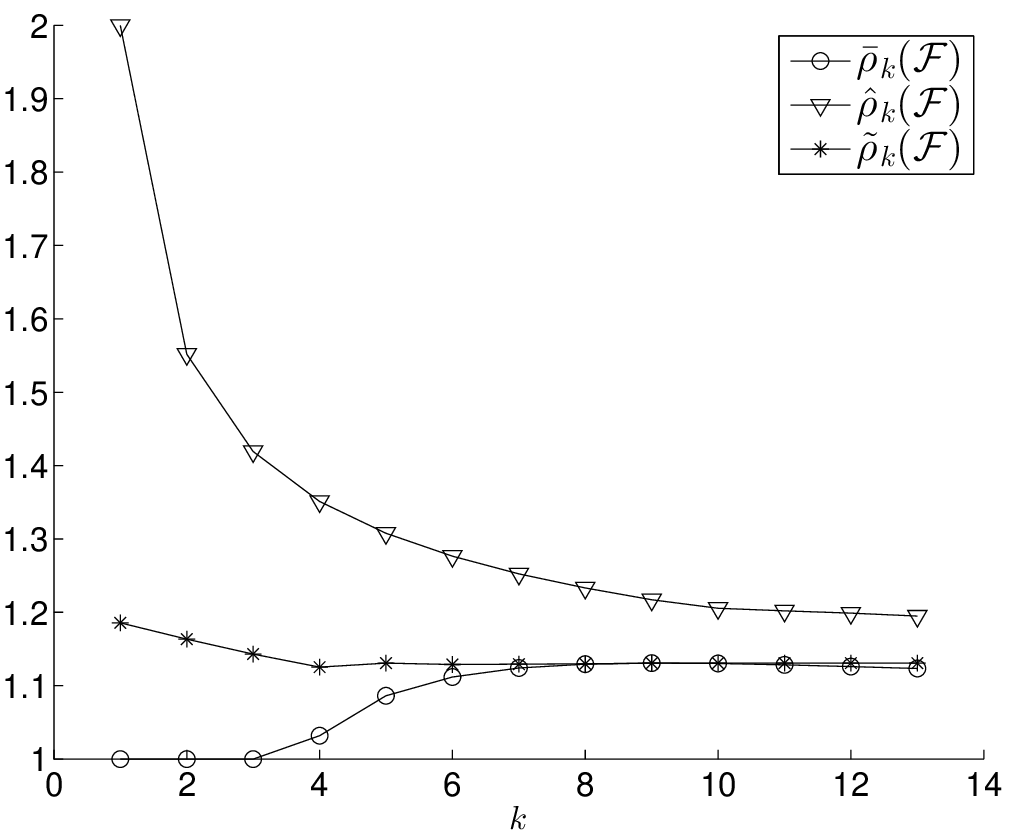}}
\label{rankone2by2}
\end{figure}
\begin{table}[!h]
\centering
\caption{The values of $\bar{\rho}_k(\F),\hat{\rho}_k(\F),\tilde{\rho}_k(\F)$ with respect to $k$.}
\begin{tabular}{|c|l|l|l|c|}
\hline
$k$&$\hat{\rho}_k(\F)$&$\bar{\rho}_k(\F)$&$\tilde{\rho}_k(\F)$&$|\tilde{\rho}_k(\F)-\rho(\F)|$\\
\hline
1&	2.000000 &	1.0000000000&	 1.1856953382& 	 5.49e-02 \\ 
2&	1.551714 &	1.0000000000&	 1.1634231348& 	 3.26e-02 \\ 
3&	1.419079 &	1.0000000000&	 1.1429810565& 	 1.22e-02 \\ 
4&	1.351138 &	1.0319129405&	 1.1253876915& 	 5.43e-03 \\ 
5&	1.307649 &	1.0861809816&	 1.1307916559& 	 2.82e-05 \\ 
6&	1.276412 &	1.1119113517&	 1.1288979596& 	 1.92e-03 \\ 
7&	1.252402 &	1.1240880182&	 1.1292099639& 	 1.61e-03 \\ 
8&	1.233124 &	1.1293241815&	 1.1297917288& 	 1.03e-03 \\ 
9&	1.217169 &	1.1308198954&	 1.1308198954& 	 0 \\ 
10&	1.205486 &	1.1302365536&	 1.1308198954& 	 0 \\ 
11&	1.201904 &	1.1284843579&	 1.1308198954& 	 0 \\ 
12&	1.198889 &	1.1260827492&	 1.1308198954& 	 0 \\ 
13&	1.194836 &	1.1233372781&	 1.1308198954& 	 0 \\ 
\hline
\end{tabular}
\label{TBrankone2by2}
 \end{table} 

\section{Concluding Remarks}
The computation of joint/generalized spectral radius has been proven to be challenging and difficult in general according to current literature since both of them are the characteristics of the worst-case operation count which usually grows faster than any polynomial in terms of matrix sizes. To identify which class of matrix families has finiteness property may lead to various efficient algorithms which can significantly reduce the computational cost.  

In this paper, we show that any finite set of matrices with at most one element's rank being greater than one possesses the finiteness property and give an explicit formula of its joint/generalized spectral radius. This result provides the possibility for using the rank-one approximation to estimate the JSR of the original matrix set.  Thus, by making use of the rank-one approximation based on singular value decomposition, we obtain some new characterizations of joint/generalized spectral radius. Numerical computations on several benchmark examples from applications show some good promises for the proposed approach.
However, we are not able to provide an estimate for the convergence rate at this point.

{\bf Acknowledgment.} The authors would like to thank Ahmadi and Parrilo for pointing out an error in section 2.1 of our earlier version posted on arXiv:1109.1356. 

\bibliographystyle{model1b-num-names}
\bibliography{refbibtex}

\end{document}